\documentclass[12pt]{amsart}

\usepackage{amsmath, amssymb, amsthm, amsfonts}
\usepackage{mathrsfs}
\usepackage{fullpage}
\usepackage{enumitem}
\usepackage{hyperref}

\theoremstyle{plain}
\newtheorem{theorem}{Theorem}[section]
\newtheorem{proposition}[theorem]{Proposition}
\newtheorem{lemma}[theorem]{Lemma}

\theoremstyle{definition}

\newtheorem{example}[theorem]{Example}

\theoremstyle{remark}

\hypersetup{
	colorlinks=true,
	linkcolor=blue,
	citecolor=blue,
	urlcolor=blue
}

\title[Cohen--Macaulay power quotients]{Cohen-Macaulay higher conormal and K\"ahler differential modules of squarefree monomial ideals}

\author{T\`ai Huy H\`a}
\address{Tulane University, Department of Mathematics, 6823 St. Charles Avenue, New Orleans LA 70118, USA}
\email{tha@tulane.edu}

\author{Nguyen Cong Minh}
\address{Faculty of Mathematics and Informatics \\ Hanoi University of Science and Technology \\
	1 Dai Co Viet \\ Hanoi, Vietnam}
\email{minh.nguyencong@hust.edu.vn}


\begin{document}

\begin{abstract}
Let $S=k[x_1,\ldots,x_n]$ and let $I=I_\Delta\subsetneq S$ be a nonzero
squarefree monomial ideal. Motivated by the classical higher-order K\"ahler
differential modules and by the theory of higher conormal modules, we study
not only the higher conormal quotients $I/I^q$, but more generally the shifted
quotients $I^r/I^q$, $1\le r<q$, in the same $I$-adic conormal filtration,
together with their symbolic analogues $I^{(r)}/I^{(q)}$. We prove that, for
every $1\le r<q$ with $q\ge3$, the module $I^r/I^q$ is Cohen--Macaulay if and
only if $I$ is a complete intersection. In sharp contrast,
$I^{(r)}/I^{(q)}$ is Cohen--Macaulay if and only if $\Delta$ is a matroid,
where loops are allowed. Thus, the Cohen--Macaulayness of a single nonexceptional
window forces the Cohen--Macaulayness of every window in the corresponding
filtration. The unique exceptional pair is $(r,q)=(1,2)$: at this conormal
level, we show that the Cohen--Macaulayness of $I/I^2$ forces $I^2=I^{(2)}$, and
hence $I/I^{(2)}$ is Cohen--Macaulay.
\end{abstract}

\maketitle

\section{Introduction}\label{sec:introduction}

Let $S=k[x_1,\ldots,x_n]$ be a polynomial ring over a field and let
$I\subseteq S$ be a homogeneous ideal. The conormal module $I/I^2$ is one of
the classical modules attached to the embedding $\operatorname{Spec}(S/I)
\subseteq \operatorname{Spec}(S)$. It is linked to K\"ahler differentials by
the conormal sequence
$$
I/I^2\longrightarrow \Omega_{S/k}\otimes_S S/I
\longrightarrow \Omega_{(S/I)/k}\longrightarrow0.
$$
The structure of $I/I^2$ has therefore played a central role in commutative
algebra and in the study of infinitesimal embeddings. Classical
complete-intersection criteria were obtained by Vasconcelos
\cite{VasconcelosRSequences} and Ferrand \cite{Ferrand}, and the homological
theory of the conormal module was developed further by Herzog
\cite{HerzogConormal}, Vasconcelos
\cite{VasconcelosHomology,VasconcelosKoszul}, and Simis--Vasconcelos
\cite{SimisVasconcelos}. See also
\cite{HeinzerKimUlrich,HunekeUlrich,VasconcelosBook} for related connections
with Rees and associated graded algebras.

There is also a classical higher-order K\"ahler theory. If $A$ is a
$k$-algebra and
$$
I_A=\ker(A\otimes_kA\longrightarrow A)
$$
is the diagonal ideal, then the module of K\"ahler differentials of order $n$
is
$$
\Omega^{(n)}_{A/k}=I_A/I_A^{n+1}.
$$
This construction goes back to Grothendieck, Nakai, and Osborn
\cite{Grothendieck,Nakai,Osborn}; see Barajas--Duarte \cite{BarajasDuarte}
for a modern treatment. Barajas--Duarte use the terminology ``module of
K\"ahler differentials of order $n$,'' while de Alba--Duarte also use
``high order K\"ahler differentials'' \cite{DeAlbaDuarte}. For a general
defining ideal $I\subseteq S$, the quotient $I/I^q$ is the corresponding
higher conormal object attached to the embedding
$\operatorname{Spec}(S/I)\subseteq\operatorname{Spec}(S)$; when $I$ is the
diagonal ideal, it recovers the classical module of K\"ahler differentials
of order $q-1$. Motivated by this higher-order K\"ahler terminology, we
refer to $I/I^q$ as the higher K\"ahler differential module associated to
$I$. Equivalently, $I/I^q$ records the first $q-1$ layers of the $I$-adic
conormal filtration.

Two classical questions about $I/I^2$ are particularly relevant here.
Vasconcelos conjectured that finite projective dimension of $I/I^2$ over
$S/I$, under the natural finite projective-dimension hypothesis on $I$, should
force $I$ to be a complete intersection. Avramov--Herzog proved the graded
characteristic-zero case \cite{AvramovHerzog}, and Briggs proved the
conjecture in full generality \cite{Briggs};
see also \cite{HerzogBriggsIyengar} for a recent account of these questions and related results concerning conormal modules and modules of differentials.
A second line of work asks what
Cohen--Macaulayness of $I/I^2$ forces on $S/I$. Results of Vasconcelos
\cite{VasconcelosKoszul}, Simis--Vasconcelos--Villarreal
\cite{SimisVasconcelosVillarreal}, Huneke--Ulrich \cite{HunekeUlrich}, and
Mantero--Xie \cite{ManteroXie} show both the strength and the limitations of
this condition.

In the Stanley--Reisner setting, the conormal question is closely related to
the Cohen--Macaulayness of the square. When $S/I$ is Cohen--Macaulay, the
exact sequence
$$
0\longrightarrow I/I^2\longrightarrow S/I^2\longrightarrow S/I\longrightarrow0
$$
and the Depth Lemma identify Cohen--Macaulayness of $I/I^2$ with that of
$S/I^2$. Thus, results on Cohen--Macaulay second powers give direct information
about conormal modules; see, for example,
\cite{FaridiHibi,MinhTrung2009,RinaldoTerai,RinaldoTeraiYoshida,TrungTuan}.
This naturally leads to higher conormal modules.

The module $I/I^q$ has a finite filtration whose successive quotients are
$I^t/I^{t+1}$, $1\le t<q$. More generally, for $1\le r<q$, the quotient
$$
I^r/I^q
$$
is the shifted part of this higher conormal filtration containing precisely
the layers $I^t/I^{t+1}$ with $r\le t<q$; in particular, it is a submodule of
$I/I^q$. Thus, once one asks for the Cohen--Macaulayness of higher conormal
modules, it is natural to ask whether the same phenomenon persists for every
finite window of the conormal filtration. Our first main theorem gives a
complete answer for squarefree monomial ideals. Theorem~\ref{thm:ordinary}
proves that for every nonzero squarefree monomial ideal $I\subsetneq S$ and
every $1\le r<q$ with $q\ge3$,
$$
I^r/I^q\text{ is Cohen--Macaulay}
\Longleftrightarrow
I\text{ is a complete intersection}.
$$
The original higher K\"ahler differential modules correspond to $r=1$.
Thus, the theorem not only classifies $I/I^q$ for every $q\ge3$, but shows that
the same complete-intersection rigidity is already detected by any shifted
window with upper exponent at least three. Writing $q=r+m$ gives, in
particular, all modules $I^r/I^{r+m}$ for fixed $r$, while taking $q=r+1$
shows that, for every $r\ge2$, the single graded piece $I^r/I^{r+1}$ is
Cohen--Macaulay if and only if $I$ is a complete intersection. Consequently,
Cohen--Macaulayness of one ordinary window with upper exponent at least three
forces Cohen--Macaulayness of every quotient $I^a/I^b$, $1\le a<b$. For edge
ideals, Theorem~\ref{CM-edge-ideals} gives the equivalent graph-theoretic
condition that, after deleting isolated vertices, the graph is a matching.

There is a natural symbolic counterpart. For a radical ideal $I$, the symbolic
power $I^{(q)}$ records the $q$-fold thickening along the minimal components of
$V(I)$ while avoiding embedded components introduced by ordinary powers. We
therefore view
$$
I/I^{(q)}
$$
as the symbolic analogue of the higher K\"ahler differential module $I/I^q$.
This interpretation is not merely formal: conormal modules and symbolic powers
already interact at the level of the second symbolic power; see Jiang
\cite{Jiang}. Moreover, $I/I^{(q)}$ has successive symbolic conormal layers
$I^{(t)}/I^{(t+1)}$, and in our previous work we classified the
Cohen--Macaulayness of these successive quotients for squarefree monomial
ideals \cite{HaMinh}. Just as in the ordinary filtration, the more general
quotient
$$
I^{(r)}/I^{(q)},\qquad 1\le r<q,
$$
is a shifted symbolic window inside $I/I^{(q)}$.

The symbolic classification is strikingly different from the ordinary one.
If $I=I_\Delta$, Theorem~\ref{thm:symbolic} proves that for every
$1\le r<q$ with $q\ge3$,
$$
I^{(r)}/I^{(q)}\text{ is Cohen--Macaulay}
\Longleftrightarrow
\Delta\text{ is a matroid}.
$$
Here matroids are allowed to have loops, so variables may belong to
$I_\Delta$. In particular, the case $r=1$ classifies all symbolic higher
conormal modules $I/I^{(q)}$, while the general statement shows that every
shifted symbolic window detects exactly the same matroid structure. Thus
Cohen--Macaulayness of a single symbolic window with upper exponent at least
three forces Cohen--Macaulayness of every symbolic window.

The two classifications parallel a well-known rigidity phenomenon for the
ambient thickenings themselves. Terai and Trung proved that, for a
Stanley--Reisner ideal, sufficiently high ordinary powers can be
Cohen--Macaulay only for complete intersections, while Cohen--Macaulay
symbolic powers force a matroid structure \cite{TeraiTrung}; the symbolic
all-powers characterization was also obtained independently in
\cite{MinhTrung,Varbaro}. Our results show that the same two rigid structures
are detected already by the corresponding conormal filtration quotients. They
fit into the exact sequences
$$
0\longrightarrow I^r/I^q\longrightarrow S/I^q\longrightarrow S/I^r\longrightarrow0
$$
and
$$
0\longrightarrow I^{(r)}/I^{(q)}\longrightarrow S/I^{(q)}
\longrightarrow S/I^{(r)}\longrightarrow0.
$$
This is not an automatic consequence of the corresponding ring-theoretic
results: Cohen--Macaulayness of the kernel in either sequence does not in
general force Cohen--Macaulayness of either neighboring term.

There is exactly one pair $1\le r<q$ not covered by $q\ge3$, namely
$(r,q)=(1,2)$. This is the classical conormal module and is genuinely
exceptional. Proposition~\ref{prop:second-symbolic} shows that if $I/I^2$ is
Cohen--Macaulay, then
$$
I^2=I^{(2)},
$$
and hence $I/I^{(2)}=I/I^2$ is Cohen--Macaulay. The exception is real. Let
$\Delta=C_5$ be a pentagon and $I=I_\Delta$. Then, $S/I$ and $S/I^2$ are
Cohen--Macaulay \cite{FaridiHibi}, so $I/I^2$ is Cohen--Macaulay; moreover,
$I^2=I^{(2)}$, and hence $I/I^{(2)}$ is Cohen--Macaulay as well. But $C_5$ is
neither a complete-intersection complex nor a matroid. The two main theorems
therefore imply that every ordinary and symbolic window with upper exponent
$q\ge3$ is non-Cohen--Macaulay for this example. Thus, the transition from the
second to the third power is genuinely sharp in both theories. See
\cite{RinaldoTeraiYoshida} for broader exceptional behavior of second powers
of Stanley--Reisner ideals.

The contrast between the ordinary and symbolic theories is already visible
for the rank-two uniform matroid $\Delta=U_{2,4}$, whose Stanley--Reisner ideal
is
$$
I_\Delta=(x_1x_2x_3,x_1x_2x_4,x_1x_3x_4,x_2x_3x_4).
$$
Here $I_\Delta$ is not a complete intersection, whereas $\Delta$ is a matroid.
Moreover, $x_1x_2x_3x_4\in I_\Delta^{(2)}\setminus I_\Delta^2$, so
Proposition~\ref{prop:second-symbolic} shows that $I_\Delta/I_\Delta^2$ is not
Cohen--Macaulay. Hence
$$
I_\Delta^r/I_\Delta^q\text{ is not Cohen--Macaulay for every }1\le r<q,
$$
whereas
$$
I_\Delta^{(r)}/I_\Delta^{(q)}\text{ is Cohen--Macaulay for every }1\le r<q.
$$
Thus, the ordinary and symbolic higher conormal filtrations can have opposite
Cohen--Macaulay behavior in every window.

The proofs of the ordinary and symbolic classifications are driven by closely
related homological obstructions. Our starting point is the relative
Hochster--Takayama formula developed in \cite{HaMinh}, which describes the
multigraded pieces of local cohomology of a monomial quotient in terms of the
relative homology of a pair of degree complexes. For a window with upper
exponent $q\ge3$, the decisive multidegree is
$$
\mathbf b=(q-1)\mathbf e_1+\mathbf e_2+\mathbf e_3.
$$
At a local three-vertex obstruction, the degree complex of the denominator
splits into two stars, while the degree complex of the numerator lies in one
of them. After introducing negative coordinates along a maximal common link,
we obtain nonzero relative homology in degree zero and therefore a nonzero
local cohomology module below the dimension. In the ordinary case,
localization at vertices reduces the general squarefree problem to locally
complete-intersection complexes, and the structure theorem of Terai and
Yoshida \cite{TeraiYoshida} completes the induction. In the symbolic case, a
failure of the matroid exchange axiom produces the same relative two-star
pattern directly.

The paper is organized as follows. Section~\ref{sec:preliminaries} records the
notation on degree complexes, the relative Hochster--Takayama formula, and a
relative zero-homology observation. Section~\ref{sec:ordinary} develops the
ordinary higher conormal filtration theory, beginning with the exceptional
second conormal module and a three-vertex obstruction, then treating edge
ideals and arbitrary squarefree monomial ideals. Section~\ref{sec:symbolic}
develops the symbolic analogue and proves the matroid characterization.

\medskip

\noindent\textbf{Acknowledgment.} The first author is partially supported by a Simons Foundation grant. The second author is partially supported by a project of Vietnam Ministry of Education and Training.

\section{Preliminaries}\label{sec:preliminaries}

Throughout the paper, $k$ is a field and $S=k[x_1,\ldots,x_n]$ with homogeneous
maximal ideal $\mathfrak m=(x_1,\ldots,x_n)$. We write $[n]=\{1,\ldots,n\}$.
If $I\subseteq S$ is squarefree, we write $I=I_\Delta$ for its
Stanley--Reisner ideal. We do not require every singleton $\{i\}$ to be a face
of $\Delta$; in the matroid terminology used later, the missing vertices are
loops. We denote by $\mathcal F(\Delta)$ the set of facets of $\Delta$ and use
$\operatorname{star}_\Delta(F)$ and $\operatorname{link}_\Delta(F)$ for the
star and link of a face $F$. We use $\varnothing$ for the void simplicial
complex and $\{\emptyset\}$ for the complex consisting only of the empty face.

For $\mathbf a=(a_1,\ldots,a_n)\in\mathbb Z^n$, set
$G_{\mathbf a}=\{i:a_i<0\}$ and let $\mathbf a^+\in\mathbb N^n$ be obtained by
replacing the negative entries of $\mathbf a$ by zero. For $F\subseteq[n]$,
write $x_F=\prod_{i\in F}x_i$ and $\mathbf e_F=\sum_{i\in F}\mathbf e_i$. If
$J$ is a monomial ideal, its degree complex at $\mathbf a$ is
$$
\Delta_{\mathbf a}(J)
=
\bigl\{F\subseteq[n]\setminus G_{\mathbf a}:
 x^{\mathbf a}\notin J S_{x_{F\cup G_{\mathbf a}}}\bigr\},
$$
where the variables indexed by $G_{\mathbf a}$ are inverted, so
$x^{\mathbf a}$ is interpreted in the indicated localization. This agrees
with Takayama's degree complex \cite{Takayama}.

We shall repeatedly use the following elementary consequence of the
definition; see also the proof of \cite[Lemma~3.1]{HaMinh} (or \cite[Theorem 1.6]{MinhTrung}).

\begin{lemma}\label{lem:negative-link}
For every monomial ideal $J\subseteq S$ and every $\mathbf a\in\mathbb Z^n$,
$$
\Delta_{\mathbf a}(J)
=
\operatorname{link}_{\Delta_{\mathbf a^+}(J)}(G_{\mathbf a}).
$$
\end{lemma}

\begin{proof}
A face $F\subseteq[n]\setminus G_{\mathbf a}$ belongs to the left-hand side
if and only if, after inverting the variables in $G_{\mathbf a}$, the monomial
$x^{\mathbf a^+}$ does not belong to $J S_{x_{F\cup G_{\mathbf a}}}$. This is
exactly the condition that $F\cup G_{\mathbf a}$ belongs to
$\Delta_{\mathbf a^+}(J)$.
\end{proof}

If $J\subseteq I$ are monomial ideals, then
$\Delta_{\mathbf a}(I)\subseteq\Delta_{\mathbf a}(J)$, and hence
$(\Delta_{\mathbf a}(J),\Delta_{\mathbf a}(I))$ is a relative simplicial
complex. We use the following relative Hochster--Takayama formula from
\cite[Theorem~3.2]{HaMinh}: whenever
$G_{\mathbf a}\in\Delta(\sqrt J)$,
$$
H^i_{\mathfrak m}(I/J)_{\mathbf a}
\cong
\widetilde H^{i-|G_{\mathbf a}|-1}
\bigl(\Delta_{\mathbf a}(J),\Delta_{\mathbf a}(I);k\bigr).
$$
Over the field $k$, nonvanishing of reduced relative homology and cohomology in
degree zero are equivalent. We shall therefore use $\widetilde H_0$ in the
combinatorial arguments below. We will also use the following elementary
relative-homology observation.

\begin{lemma}\label{lem:relative-zero}
	Let $X=A\cup B$ be a simplicial complex such that $A$ and $B$ each contain a
	vertex and $A\cap B=\{\emptyset\}$. If $C\subseteq A$ is a simplicial
	subcomplex, then $\widetilde H_0(X,C;k)\neq0$.
\end{lemma}

\begin{proof} This observation is taken from \cite[Lemma 3.1]{MinhNakamura}. For the reader's convenience, we provide a direct proof here. No nonempty face of $X$ meets both $A$ and $B$, so no connected component of
	$X$ meets both subcomplexes. If $C$ has no vertex, then $X$ is disconnected and
	$\widetilde H_0(X;k)\neq0$. Since $\widetilde H_0(C;k)=0$, the long exact
	sequence of the pair gives an injection
	$\widetilde H_0(X;k)\to\widetilde H_0(X,C;k)$, so the latter group is nonzero.
	If $C$ has a vertex, any connected component contained in $B$ is disjoint from
	$C$ and determines a nonzero class in $H_0(X,C;k)$. Since $C$ is nonempty in
	this case, ordinary and reduced relative homology agree in degree zero.
\end{proof}

\section{Ordinary higher conormal filtration}\label{sec:ordinary}

For the rest of the paper, unless stated otherwise, let
$I=I_\Delta\subsetneq S$ be a nonzero squarefree monomial ideal. We first
record the support calculation for arbitrary windows in both filtrations.

\begin{lemma}\label{lem:support}
Let $1\le r<q$. Then
$$
\operatorname{Supp}(I^r/I^q)
=
\operatorname{Supp}(I^{(r)}/I^{(q)})
=
V(I).
$$
Consequently,
$\dim(I^r/I^q)=\dim(I^{(r)}/I^{(q)})=\dim(S/I)$.
\end{lemma}

\begin{proof} Let $Q \in \operatorname{Spec } S \setminus V(I)$. Then, $I \not\subseteq Q$. Choose $f\in I\setminus Q$. Since
$f^r\in I^r\subseteq I^{(r)}$ and $f^q\in I^q\subseteq I^{(q)}$,
localization at $Q$ makes each of $I^r$, $I^q$, $I^{(r)}$, and $I^{(q)}$
equal to $S_Q$. Thus, both quotients vanish after localization at $Q$.

Conversely, let $P\in\operatorname{Min}(I)$. Since $I$ is radical,
$I_P=PS_P$, while symbolic powers commute with localization at a minimal
prime. Hence,
$$
(I^r/I^q)_P\cong P^rS_P/P^qS_P
$$
and
$$
(I^{(r)}/I^{(q)})_P\cong P^rS_P/P^qS_P.
$$
These modules are nonzero because $PS_P$ is a nonzero maximal ideal of the regular local ring $S_P$ and its powers are strict. Thus, every minimal prime of $I$
belongs to both supports. Since support is specialization closed, both supports
are $V(I)$.
\end{proof}

The only pair $1\le r<q$ with $q<3$ is $(r,q)=(1,2)$. At this exceptional
level, ordinary Cohen--Macaulayness forces the ordinary and symbolic squares
to coincide.

\begin{proposition}\label{prop:second-symbolic}
Let $I\subsetneq S$ be a nonzero squarefree monomial ideal. If $I/I^2$ is
Cohen--Macaulay, then
$$
I^2=I^{(2)}.
$$
In particular, $I/I^{(2)}$ is Cohen--Macaulay.
\end{proposition}

\begin{proof}
Set $M=I/I^2$ and $d=\dim(S/I)$. By Lemma~\ref{lem:support},
$\operatorname{Supp}(M)=V(I)$ and $\dim M=d$. Since $M$ is a finitely
generated Cohen--Macaulay $S$-module, it is unmixed; equivalently,
$\operatorname{Ass}_S(M)=\operatorname{Assh}_S(M)$. Hence,
$$
\dim S/Q=d\text{ for every }Q\in\operatorname{Ass}_S(M).
$$
Every minimal prime of $I$ is a minimal prime of $\operatorname{Supp}(M)$,
and hence belongs to $\operatorname{Ass}_S(M)$. We claim that
$$
\operatorname{Ass}_S(M)=\operatorname{Min}(I).
$$
Indeed, let $Q\in\operatorname{Ass}_S(M)$. Since $Q\in V(I)$, there is
$P\in\operatorname{Min}(I)$ with $P\subseteq Q$. Both $P$ and $Q$ belong
to $\operatorname{Ass}_S(M)$, so $\dim S/P=\dim S/Q=d$. A strict
inclusion of primes in $S$ strictly decreases dimension, and therefore $P=Q$.

Set $T=I^{(2)}/I^2\subseteq M$. For every $P\in\operatorname{Min}(I)$,
radicality gives $I_P=PS_P$, and symbolic localization gives
$$
(I^{(2)})_P=P^2S_P=(I_P)^2=(I^2)_P.
$$
Thus, $T_P=0$ for every $P\in\operatorname{Min}(I)$. If $T\neq0$, choose
$Q\in\operatorname{Ass}_S(T)$. Since $T\subseteq M$,
$\operatorname{Ass}_S(T)\subseteq\operatorname{Ass}_S(M)=\operatorname{Min}(I)$.
Hence, $Q$ is a minimal prime of $I$, but $Q\in\operatorname{Ass}_S(T)$
implies $T_Q\neq0$, a contradiction. Therefore $T=0$, so
$I^2=I^{(2)}$. The final assertion follows because then
$I/I^{(2)}=I/I^2$.
\end{proof}

For upper exponent at least three, the following local configuration is the
basic obstruction for every ordinary window.

\begin{lemma}\label{lem:three-vertex-obstruction}
Let $1\le r<q$ with $q\ge3$. Suppose that, after relabeling three vertices
of $\Delta$,
$$
\{1\},\{2,3\}\in\Delta
\text{ and }
\{1,2\},\{1,3\}\notin\Delta.
$$
Then, $I^r/I^q$ is not Cohen--Macaulay.
\end{lemma}

\begin{proof}
Set $\mathbf b=(q-1)\mathbf e_1+\mathbf e_2+\mathbf e_3$. We first claim
that
$$
\Delta_{\mathbf b}(I^q)
=
\operatorname{star}_\Delta(1)
\cup
\operatorname{star}_\Delta(\{2,3\}).
$$
Let $F\in\operatorname{star}_\Delta(1)$. If
$x^{\mathbf b}=x_1^{q-1}x_2x_3$ belonged to $I^qS_{x_F}$, then there would
be $q$ minimal monomial generators of $I$ whose product divides
$x^{\mathbf b}$ after the variables indexed by $F$ are declared units.
Since $F\in\Delta$, none of these generators can become a unit. Moreover,
the noninvertible part of each generator must be divisible by $x_2$ or $x_3$;
otherwise its support would give a nonface contained in $F\cup\{1\}$. This is impossible because
$x^{\mathbf b}$ has total degree two in $x_2,x_3$ and $q\ge3$. Hence
$F\in\Delta_{\mathbf b}(I^q)$.

If $F\in\operatorname{star}_\Delta(\{2,3\})$, the same argument shows
that the noninvertible part of each of the $q$ factors must be divisible by
$x_1$. Their product would then be divisible by $x_1^q$, contradicting the
exponent $q-1$ of $x_1$ in $x^{\mathbf b}$. Thus, the second star is also
contained in $\Delta_{\mathbf b}(I^q)$.

Conversely, suppose that $F$ belongs to neither star. If $F\notin\Delta$,
then $IS_{x_F}=S_{x_F}$, so $F\notin\Delta_{\mathbf b}(I^q)$. Assume
$F\in\Delta$. Since $F\cup\{1\}\notin\Delta$, a minimal nonface
contained in $F\cup\{1\}$ contains $1$, and localization at $F$ gives
$x_1\in IS_{x_F}$. Similarly, since
$F\cup\{2,3\}\notin\Delta$, one of $x_2,x_3,x_2x_3$ belongs to
$IS_{x_F}$. Taking $q-1$ copies of $x_1$ and one copy of this last monomial
shows that $x^{\mathbf b}\in I^qS_{x_F}$. This proves the claimed equality.

We next claim that
$$
\Delta_{\mathbf b}(I^r)\subseteq\operatorname{star}_\Delta(1).
$$
Indeed, let $F\notin\operatorname{star}_\Delta(1)$. If
$F\notin\Delta$, then $I^rS_{x_F}=S_{x_F}$. If $F\in\Delta$, the same
minimal-nonface argument gives $x_1\in IS_{x_F}$, so
$x_1^r\in I^rS_{x_F}$. Since $r\le q-1$, the monomial $x_1^r$ divides
$x^{\mathbf b}$. In either case $F\notin\Delta_{\mathbf b}(I^r)$,
proving the inclusion.

Set
$$
L=\operatorname{link}_\Delta(1)
\cap
\operatorname{link}_\Delta(\{2,3\}),
$$
choose a maximal face $U$ of $L$, and put
$\mathbf a=\mathbf b-\sum_{i\in U}\mathbf e_i$. Then
$G_{\mathbf a}=U$ and $\mathbf a^+=\mathbf b$. By
Lemma~\ref{lem:negative-link},
$$
\Delta_{\mathbf a}(I^q)=A\cup B,
$$
where
$$
A=\operatorname{link}_{\operatorname{star}_\Delta(1)}(U)
\text{ and }
B=\operatorname{link}_{\operatorname{star}_\Delta(\{2,3\})}(U).
$$
The complex $A$ contains the vertex $1$, while $B$ contains the edge
$\{2,3\}$. Moreover, $A\cap B=\{\emptyset\}$. Indeed, if $W$ were a
nonempty face of $A\cap B$, then
$U\cup W\cup\{1\}\in\Delta$ forces $2,3\notin W$, while
$U\cup W\cup\{2,3\}\in\Delta$ forces $1\notin W$. Hence,
$U\cup W\in L$, contradicting the maximality of $U$.

The inclusion above and Lemma~\ref{lem:negative-link} give
$\Delta_{\mathbf a}(I^r)\subseteq A$. Therefore,
Lemma~\ref{lem:relative-zero} yields
$$
\widetilde H_0
\bigl(\Delta_{\mathbf a}(I^q),\Delta_{\mathbf a}(I^r);k\bigr)
\neq0.
$$
Since $U\in\Delta=\Delta(\sqrt{I^q})$, the relative
Hochster--Takayama formula applies and gives
$$
H_{\mathfrak m}^{|U|+1}(I^r/I^q)_{\mathbf a}\neq0.
$$
Since $U\cup\{2,3\}\in\Delta$, one has
$|U|+2\le\dim\Delta+1$. By Lemma~\ref{lem:support},
$$
|U|+1\le\dim\Delta<\dim\Delta+1=\dim(I^r/I^q).
$$
Thus, $I^r/I^q$ is not Cohen--Macaulay.
\end{proof}

For edge ideals, the obstruction gives a concrete classification for all
power-filtration windows.

\begin{theorem}\label{CM-edge-ideals}
Let $G$ be a finite simple graph with at least one edge, let
$S=k[x_1,\ldots,x_n]$, and let $I=I(G)$. Fix integers $1\le r<q$ with
$q\ge3$. Then, the following are equivalent:
\begin{enumerate}
\item $I^r/I^q$ is Cohen--Macaulay;
\item $I$ is a complete intersection;
\item after removing isolated vertices, $G$ is a matching.
\end{enumerate}
Consequently, if $I^r/I^q$ is Cohen--Macaulay for one such pair $(r,q)$,
then $I^a/I^b$ is Cohen--Macaulay for every $1\le a<b$.
\end{theorem}

\begin{proof}
Let $\Delta=\operatorname{Ind}(G)$ and set $M=I^r/I^q$. Assume first that
$M$ is Cohen--Macaulay. If some connected component of $G$ were not complete,
it would contain an induced path $2-1-3$. Then,
$$
\{1\},\{2,3\}\in\Delta
\text{ and }
\{1,2\},\{1,3\}\notin\Delta,
$$
contradicting Lemma~\ref{lem:three-vertex-obstruction}. Hence, every connected
component of $G$ is complete. After removing isolated vertices, write
$$
G=K_{r_1}\sqcup\cdots\sqcup K_{r_s},\text{ } r_i\ge2.
$$
All minimal vertex covers of $G$ have cardinality
$\sum_{i=1}^s(r_i-1)$, so $I$ is unmixed. By Lemma~\ref{lem:support},
$\operatorname{Supp}(M)=V(I)$ and $\dim M=\dim(S/I)$. Since $M$ is
Cohen--Macaulay, every associated prime of $M$ has dimension $\dim(S/I)$.
Moreover, every minimal prime of $I$ is a minimal prime of
$\operatorname{Supp}(M)$ and hence belongs to $\operatorname{Ass}_S(M)$.
If $Q\in\operatorname{Ass}_S(M)$ and $P\in\operatorname{Min}(I)$ satisfies
$P\subseteq Q$, then $\dim S/P=\dim S/Q=\dim(S/I)$, so $P=Q$. Therefore,
$$
\operatorname{Ass}_S(M)=\operatorname{Min}(I).
$$

Suppose that some component $C$ is $K_\ell$ with $\ell\ge3$, and choose
distinct vertices $x,y,z$ of $C$. For every other nontrivial complete component, omit
one vertex, and let $P$ be the monomial prime generated by all vertices of
$C$ together with all nonomitted vertices in the other complete components.
No isolated vertex is included in $P$. After localizing at $P$,
$$
I_P=J+L,
$$
where $J=I(K_\ell)$ and $L$ is generated by variables from the other complete
components. The variables supporting $J$ and $L$ are disjoint, and $PS_P$ is
the maximal ideal generated by the variables of $C$ and those generating $L$.

Assume first that $L\neq0$, choose a variable $w\in L$, and set
$$
u=xyz\,w^{q-2}.
$$
Since $xyz\in J$, one has
$u\in JL^{q-2}\subseteq I_P^{q-1}\subseteq I_P^r$. On the other hand,
$u\notin I_P^q$. Indeed,
$$
(J+L)^q=\sum_{a+b=q}J^aL^b.
$$
If $u\in J^aL^b$, then $b\le q-2$, and hence $a\ge2$. Since the
variables supporting $J$ and $L$ are disjoint, this would force the $C$-part
$xyz$ of $u$ to belong to $J^a\subseteq J^2$, which is impossible because
every monomial in $J^2$ has degree at least four.

For every variable $v$ of $C$, one has $vxyz\in J^2$, so
$vu\in J^2L^{q-2}\subseteq I_P^q$. If $t$ is a variable generating $L$,
then $xyz\in J$ and $tw^{q-2}\in L^{q-1}$, so
$tu\in JL^{q-1}\subseteq I_P^q$. Hence,
$I_P^q:u$ contains the maximal ideal $PS_P$. Since $u\notin I_P^q$, this
colon is proper, and therefore
$$
I_P^q:u=PS_P.
$$
Thus, the class of $u$ in $M_P=I_P^r/I_P^q$ has annihilator $PS_P$, so
$PS_P\in\operatorname{Ass}_{S_P}(M_P)$.

If $L=0$, set
$$
u=(xy)^{q-2}xyz.
$$
Since $xyz\in J$, one has $u\in J^{q-1}\subseteq J^r$. Its degree is
$2q-1$, whereas every monomial in $J^q$ has degree at least $2q$, so
$u\notin J^q$. For every variable $v$ of $C$,
$$
xu=(xy)^{q-1}(xz),\text{ } yu=(xy)^{q-1}(yz),
$$
$$
zu=(xy)^{q-2}(xz)(yz),
$$
and, if $v\notin\{x,y,z\}$,
$$
vu=(xy)^{q-2}(xv)(yz).
$$
Thus, $J^q:u=PS_P$, and again $PS_P\in\operatorname{Ass}_{S_P}(M_P)$.

By the localization formula for associated primes, there is
$Q\in\operatorname{Ass}_S(M)$ with $Q\subseteq P$ and
$QS_P=PS_P$. Contracting to $S$ gives $Q=P$, so
$P\in\operatorname{Ass}_S(M)$. But $P$ is not minimal over $I$,
because it contains every vertex of $C$, whereas a minimal vertex cover of
$C=K_\ell$ contains only $\ell-1$ vertices. This contradicts
$\operatorname{Ass}_S(M)=\operatorname{Min}(I)$. Hence, every nontrivial
component of $G$ is $K_2$, proving $(1)\Rightarrow(3)$.

The equivalence of (2) and (3) follows from the characterization of monomial
complete intersections by pairwise disjoint supports of their minimal
generators. If $I$ is a complete intersection, then
$\operatorname{gr}_I(S)\cong\operatorname{Sym}_{S/I}(I/I^2)$; in particular,
each $I^t/I^{t+1}$ is a finite free $S/I$-module. It follows inductively from
$0\to I^t/I^{t+1}\to S/I^{t+1}\to S/I^t\to0$ that $S/I^t$ is
Cohen--Macaulay of dimension $\dim(S/I)$ for every $t\ge1$. The exact
sequence
$$
0\longrightarrow I^r/I^q\longrightarrow S/I^q\longrightarrow S/I^r
\longrightarrow0
$$
together with the Depth Lemma and Lemma~\ref{lem:support} shows that
$I^r/I^q$ is Cohen--Macaulay. This
proves $(2)\Rightarrow(1)$. The final assertion follows because a complete
intersection has Cohen--Macaulay quotients $I^a/I^b$ for all $1\le a<b$.
\end{proof}

We prepare the passage from edge ideals to arbitrary squarefree monomial
ideals. The next lemma removes linear generators from an arbitrary window,
while the following one reduces localizations to Stanley--Reisner ideals of
vertex links.

\begin{lemma}\label{lem:linear-generators}
Let $A$ be a polynomial ring over $k$, let $R=A[y_1,\ldots,y_s]$, and let
$0\neq K\subsetneq A$ be a squarefree monomial ideal. Set
$J=(y_1,\ldots,y_s)+K$. Fix $1\le r<q$. If $J^r/J^q$ is Cohen--Macaulay as
an $R$-module, then $K^r/K^q$ is Cohen--Macaulay as an $A$-module.
\end{lemma}

\begin{proof}
Set $M=J^r/J^q$, let $\mathfrak n$ be the homogeneous maximal ideal of $A$,
and set $\mathfrak r=\mathfrak nR+(y_1,\ldots,y_s)$. Since
$y_i^{q-r}J^r\subseteq J^q$, each $y_i$ acts nilpotently on $M$. Hence, $M$
is finitely generated as an $A$-module and
$$
\sqrt{\mathfrak nR+\operatorname{Ann}_R(M)}
=
\sqrt{\mathfrak r+\operatorname{Ann}_R(M)}.
$$
The independence theorem gives
$H_{\mathfrak nR}^i(M)\cong H_{\mathfrak r}^i(M)$ for every $i$.
Since $M$ is finite over $A$, one also has
$H_{\mathfrak n}^i(M)\cong H_{\mathfrak nR}^i(M)$. By
Lemma~\ref{lem:support}, applied in $R$,
$$
\dim_RM=\dim(R/J)=\dim(A/K)=:d.
$$
Since $M$ is Cohen--Macaulay, these local cohomology modules vanish for
$i<d$.

Give $R$ the $\mathbb N^s$-grading induced by the variables
$y_1,\ldots,y_s$. For every $t\ge1$, specialization at
$y_1=\cdots=y_s=0$ gives $J^t\cap A=K^t$. Hence, the component of $M$ of
$y$-degree zero is
$$
M_{\mathbf0}=K^r/K^q.
$$
Projection onto $M_{\mathbf0}$ is $A$-linear, so $K^r/K^q$ is an
$A$-module direct summand of $M$. Therefore
$H_{\mathfrak n}^i(K^r/K^q)$ vanishes for every $i<d$. By
Lemma~\ref{lem:support}, $\dim_A(K^r/K^q)=d$, and hence $K^r/K^q$ is
Cohen--Macaulay.
\end{proof}

\begin{lemma}\label{lem:localization}
Assume that $I$ has no minimal generator of degree one. Fix $i\in[n]$ and
set
$$
V_i=\{j\in[n]\setminus\{i\}:\{i,j\}\in\Delta\},
\text{ }
W_i=[n]\setminus(V_i\cup\{i\}).
$$
Let $A_i=k[x_j:j\in V_i]$ and
$K_i=I_{\operatorname{link}_\Delta(i)}\subseteq A_i$. Then
$$
IS_{x_i}=\bigl(K_i+(x_j:j\in W_i)\bigr)S_{x_i}.
$$
Fix $1\le r<q$. If $I^r/I^q$ is Cohen--Macaulay, then either $K_i=0$, or
$K_i^r/K_i^q$ is Cohen--Macaulay.
\end{lemma}

\begin{proof}
Set $R_i=k[x_j:j\neq i]$ and
$J_i=K_i+(x_j:j\in W_i)\subseteq R_i$. We first prove
$IS_{x_i}=J_iS_{x_i}$. If $j\in W_i$, then $\{i,j\}\notin\Delta$, so
$x_ix_j\in I$ and $x_j\in IS_{x_i}$. If $F\subseteq V_i$ is a nonface of
$\operatorname{link}_\Delta(i)$, then $F\cup\{i\}\notin\Delta$, so
$x_ix_F\in I$ and $x_F\in IS_{x_i}$. Hence,
$J_iS_{x_i}\subseteq IS_{x_i}$.

Conversely, let $x_F$ be a squarefree monomial generator of $I$. If
$F\cap W_i\neq\emptyset$, then $x_F\in J_iS_{x_i}$. Suppose
$F\subseteq V_i\cup\{i\}$. If $i\in F$, then $F\setminus\{i\}$ is a
nonface of the link, so $x_F\in K_iS_{x_i}$. If $i\notin F$, then
$F\notin\Delta$ and therefore $F\cup\{i\}\notin\Delta$; hence $F$ is
a nonface of the link and again $x_F\in K_iS_{x_i}$. This proves the
asserted equality.

Consequently,
$$
(I^r/I^q)_{x_i}
\cong
(J_i^r/J_i^q)\otimes_{R_i}S_{x_i}.
$$
Assume that $I^r/I^q$ is Cohen--Macaulay. If $J_i=0$, then $K_i=0$. Suppose
$J_i\neq0$. Let $\mathfrak n_i=(x_j:j\neq i)\subseteq R_i$ and
$Q_i=\mathfrak n_iS_{x_i}$. The map
$(R_i)_{\mathfrak n_i}\to(S_{x_i})_{Q_i}$ is flat and local with closed
fiber $k(x_i)$. Cohen--Macaulayness localizes, so the corresponding localization of
$(I^r/I^q)_{x_i}$ is Cohen--Macaulay. Since this module is the base change of
$(J_i^r/J_i^q)_{\mathfrak n_i}$, and the closed fiber has dimension and
depth zero, the depth and dimension formulas show that
$(J_i^r/J_i^q)_{\mathfrak n_i}$ is Cohen--Macaulay. Since
$J_i^r/J_i^q$ is a finitely generated graded $R_i$-module, it is
Cohen--Macaulay. If $K_i\neq0$, Lemma~\ref{lem:linear-generators} applied to
$J_i=(x_j:j\in W_i)+K_i$ gives that $K_i^r/K_i^q$ is Cohen--Macaulay.
\end{proof}

We can now combine localization, the three-vertex obstruction, and the
structure of locally complete-intersection complexes.

\begin{theorem}\label{thm:ordinary}
Let $I\subsetneq S$ be a nonzero squarefree monomial ideal and let
$1\le r<q$ with $q\ge3$. Then
$$
I^r/I^q\text{ is Cohen--Macaulay}
\Longleftrightarrow
I\text{ is a complete intersection}.
$$
Consequently, if $I^r/I^q$ is Cohen--Macaulay for one such pair $(r,q)$,
then $I^a/I^b$ is Cohen--Macaulay for every $1\le a<b$.
\end{theorem}

\begin{proof}
Suppose first that $I$ is a complete intersection. As recalled in the proof of
Theorem~\ref{CM-edge-ideals}, the associated graded ring
$\operatorname{gr}_I(S)\cong\operatorname{Sym}_{S/I}(I/I^2)$, and hence
$S/I^t$ is Cohen--Macaulay of dimension $\dim(S/I)$ for every $t\ge1$. From
$$
0\longrightarrow I^r/I^q\longrightarrow S/I^q\longrightarrow S/I^r
\longrightarrow0
$$
and Lemma~\ref{lem:support}, the Depth Lemma gives that $I^r/I^q$ is
Cohen--Macaulay.

Conversely, suppose that $I^r/I^q$ is Cohen--Macaulay. We proceed by
induction on $n$, assuming the assertion for all squarefree monomial ideals in
fewer than $n$ variables. The assertion is immediate for $n\le2$. First
suppose that $I$ has no linear minimal generator. For each $i\in[n]$, let
$K_i=I_{\operatorname{link}_\Delta(i)}$ as in Lemma~\ref{lem:localization}.
That lemma shows that either $K_i=0$, or $K_i^r/K_i^q$ is Cohen--Macaulay.
Since the link has fewer vertices, the induction hypothesis implies that every
$K_i$ is a complete intersection. Thus, $\Delta$ is locally a
complete-intersection complex in the sense of Terai and Yoshida.

We claim that $\Delta$ is connected whenever $\dim\Delta\ge1$. Otherwise,
choose an edge $\{2,3\}$ in one connected component and a vertex $1$ in
another. Then
$$
\{1\},\{2,3\}\in\Delta
\text{ and }
\{1,2\},\{1,3\}\notin\Delta,
$$
contradicting Lemma~\ref{lem:three-vertex-obstruction}. Hence, $\Delta$ is
connected.

If $\dim\Delta\ge2$, then \cite[Theorem~1.5]{TeraiYoshida} implies that
a connected locally complete-intersection complex is a complete-intersection
complex. Hence, $I$ is a complete intersection.

Suppose $\dim\Delta=1$. The link of each vertex is a zero-dimensional
complete-intersection complex. A zero-dimensional complex on $t$ vertices has
Stanley--Reisner ideal $(x_ax_b:1\le a<b\le t)$, which is a complete
intersection only when $t\le2$. Hence, every vertex of the graph $\Delta$ has
degree at most two. Since $\Delta$ is connected, it is a path or a cycle. If
$\Delta$ is the $3$-cycle, then $I$ is principal. Otherwise $\Delta$ is
triangle-free, every minimal nonface has cardinality two, and $I$ is the edge
ideal of the complement of the one-skeleton of $\Delta$. Theorem~\ref{CM-edge-ideals}
then implies that $I$ is a complete intersection.

If $\dim\Delta=0$, then $I$ is the edge ideal of the complete graph on the
vertices of $\Delta$, and Theorem~\ref{CM-edge-ideals} again implies that
$I$ is a complete intersection.

This proves the result when $I$ has no linear minimal generator. In general,
let $y_1,\ldots,y_s$ be the variables belonging to $I$, let $A$ be the
polynomial ring on the remaining variables, and write
$I=(y_1,\ldots,y_s)+K$, where $K\subseteq A$ has no linear minimal
generator. If $K=0$, then $I$ is generated by variables. If $K\neq0$,
Lemma~\ref{lem:linear-generators} gives that $K^r/K^q$ is Cohen--Macaulay,
so the case already proved shows that $K$ is a complete intersection. Since
the variables $y_1,\ldots,y_s$ do not occur in the minimal generators of
$K$, the ideal $I$ is a complete intersection.

The final assertion follows from the forward implication for complete
intersections, which applies to every pair $1\le a<b$.
\end{proof}

The upper exponent $q=2$ is genuinely exceptional, even in arbitrarily high
dimension. The following family makes this explicit.

\begin{example}\label{ex:stellar-family}
Fix $d\ge2$ and let
$$
R=k[x_1,\ldots,x_d,y_1,\ldots,y_d,v].
$$
Let $\Gamma$ be the complete-intersection complex on
$\{x_1,\ldots,x_d,y_1,\ldots,y_d\}$ with Stanley--Reisner ideal
$$
I_\Gamma=(x_1y_1,\ldots,x_dy_d).
$$
Thus, $\Gamma$ is the boundary complex of the $d$-dimensional cross-polytope
and, in particular, is a non-acyclic complete-intersection complex. Let
$\Delta$ be the stellar subdivision of $\Gamma$ at the facet
$\{x_1,\ldots,x_d\}$, with new vertex $v$. Its minimal nonfaces give
$$
I_\Delta=(x_1y_1,\ldots,x_dy_d,vy_1,\ldots,vy_d,x_1x_2\cdots x_d).
$$
By \cite[Theorem~5.4]{RinaldoTeraiYoshida}, $R/I_\Delta^2$ is
Cohen--Macaulay. Since stellar subdivision of a simplicial sphere is again a
simplicial sphere, $R/I_\Delta$ is Cohen--Macaulay. The exact sequence
$$
0\longrightarrow I_\Delta/I_\Delta^2
\longrightarrow R/I_\Delta^2
\longrightarrow R/I_\Delta\longrightarrow0
$$
together with the Depth Lemma and Lemma~\ref{lem:support} shows that
$I_\Delta/I_\Delta^2$ is Cohen--Macaulay, and
Proposition~\ref{prop:second-symbolic} gives
$I_\Delta^2=I_\Delta^{(2)}$.

For $d\ge2$, the minimal generators $x_1y_1$ and $vy_1$ have intersecting
supports, so $I_\Delta$ is not a complete intersection. Hence,
Theorem~\ref{thm:ordinary} gives
$$
I_\Delta^r/I_\Delta^q\text{ is not Cohen--Macaulay whenever }
1\le r<q\text{ and }q\ge3.
$$
Thus, the exceptional behavior at the second conormal level is not confined to
graphs or low dimension. When $d=2$, this construction recovers the pentagon.
\end{example}

\section{Symbolic higher conormal filtration}\label{sec:symbolic}

We now turn to symbolic windows. For a squarefree monomial ideal
$I=I_\Delta$, recall that
$$
I^{(q)}=\bigcap_{F\in\mathcal F(\Delta)}P_F^q,
\text{ }
P_F=(x_i:i\notin F).
$$
The following description of the corresponding degree complexes is known;
see \cite{MinhTrung,Takayama}. We record it for later use.

\begin{lemma}\label{lem:symbolic-degree}
Let $\mathbf c\in\mathbb N^n$ and $q\ge1$. Then
$$
\Delta_{\mathbf c}(I^{(q)})
=
\left\langle
F\in\mathcal F(\Delta):
\sum_{i\notin F}c_i\le q-1
\right\rangle.
$$
\end{lemma}
To detect failure of the matroid condition, we use the following standard
three-element witness; see \cite[Theorem~39.1]{Schrijver}.

\begin{lemma}\label{lem:matroid-witness}
The simplicial complex $\Delta$ is not a matroid if and only if there exist
faces $F,G\in\Delta$ such that, after relabeling,
$$
F\setminus G=\{1\},
\text{ }
G\setminus F=\{2,3\},
$$
and $F\cup\{2\},F\cup\{3\}\notin\Delta$.
\end{lemma}

The degree-complex formula and this witness turn a failure of matroid exchange
into the same relative two-star pattern as in the ordinary theory.

\begin{theorem}\label{thm:symbolic}
Let $I=I_\Delta\subsetneq S$ be a nonzero squarefree monomial ideal and let
$1\le r<q$ with $q\ge3$. Then
$$
I^{(r)}/I^{(q)}\text{ is Cohen--Macaulay}
\Longleftrightarrow
\Delta\text{ is a matroid}.
$$
Here matroids are allowed to have loops. Consequently, if one symbolic window
$I^{(r)}/I^{(q)}$ with $q\ge3$ is Cohen--Macaulay, then
$I^{(a)}/I^{(b)}$ is Cohen--Macaulay for every $1\le a<b$.
\end{theorem}

\begin{proof}
Suppose first that $\Delta$ is a matroid. By \cite{MinhTrung,Varbaro},
$S/I^{(t)}$ is Cohen--Macaulay of dimension $\dim(S/I)$ for every $t\ge1$.
The exact sequence
$$
0\longrightarrow I^{(r)}/I^{(q)}
\longrightarrow S/I^{(q)}
\longrightarrow S/I^{(r)}\longrightarrow0
$$
together with the Depth Lemma and Lemma~\ref{lem:support} shows that
$I^{(r)}/I^{(q)}$ is Cohen--Macaulay.

Conversely, suppose that $I^{(r)}/I^{(q)}$ is Cohen--Macaulay and that
$\Delta$ is not a matroid. By Lemma~\ref{lem:matroid-witness}, after
relabeling there are faces $F,G\in\Delta$ such that
$F\setminus G=\{1\}$, $G\setminus F=\{2,3\}$, and
$F\cup\{2\},F\cup\{3\}\notin\Delta$. Set
$$
\mathbf b=(q-1)\mathbf e_1+\mathbf e_2+\mathbf e_3.
$$
We claim that
$$
\Delta_{\mathbf b}(I^{(q)})
=
\operatorname{star}_\Delta(1)
\cup
\operatorname{star}_\Delta(\{2,3\})
$$
and
$$
\Delta_{\mathbf b}(I^{(r)})
\subseteq
\operatorname{star}_\Delta(1).
$$
By Lemma~\ref{lem:symbolic-degree}, a facet $H$ contributes to
$\Delta_{\mathbf b}(I^{(q)})$ exactly when
$\sum_{i\notin H}b_i\le q-1$. If $1\in H$, the sum is at most two, and
$2\le q-1$. If $1\notin H$, the sum is at least $q-1$, with equality
precisely when $2,3\in H$. This proves the first equality.

For the numerator, a contributing facet satisfies
$\sum_{i\notin H}b_i\le r-1$. If $1\notin H$, the left side is at least
$q-1>r-1$. Hence, every contributing facet contains $1$, proving the second
inclusion.

Set
$$
L=\operatorname{link}_\Delta(1)
\cap
\operatorname{link}_\Delta(\{2,3\}).
$$
Since $F\cap G=F\setminus\{1\}=G\setminus\{2,3\}$, the face
$F\cap G$ belongs to $L$. Choose a maximal face $U$ of $L$ containing
$F\cap G$, and set
$\mathbf a=\mathbf b-\sum_{i\in U}\mathbf e_i$. Then
$G_{\mathbf a}=U$ and $\mathbf a^+=\mathbf b$. By
Lemma~\ref{lem:negative-link},
$$
\Delta_{\mathbf a}(I^{(q)})=A\cup B,
$$
where
$$
A=\operatorname{link}_{\operatorname{star}_\Delta(1)}(U)
\text{ and }
B=\operatorname{link}_{\operatorname{star}_\Delta(\{2,3\})}(U).
$$
The complex $A$ contains the vertex $1$, while $B$ contains the edge
$\{2,3\}$.

We claim that $A\cap B=\{\emptyset\}$. Suppose $W$ is a nonempty face of
$A\cap B$. Since $U\cup W\cup\{1\}\in\Delta$, membership $2\in W$
would imply
$$
F\cup\{2\}=(F\cap G)\cup\{1,2\}
\subseteq U\cup W\cup\{1\}\in\Delta,
$$
a contradiction; similarly $3\notin W$. Since
$U\cup W\cup\{2,3\}\in\Delta$, membership $1\in W$ would also
imply $F\cup\{2\}\in\Delta$. Thus, $W$ is disjoint from
$\{1,2,3\}$, so $U\cup W\in L$, contradicting maximality of $U$.

By Lemma~\ref{lem:negative-link} and the numerator inclusion,
$$
\Delta_{\mathbf a}(I^{(r)})\subseteq A.
$$
Lemma~\ref{lem:relative-zero} gives
$$
\widetilde H_0
\bigl(\Delta_{\mathbf a}(I^{(q)}),
\Delta_{\mathbf a}(I^{(r)});k\bigr)\neq0.
$$
Since $U\in\Delta=\Delta(\sqrt{I^{(q)}})$, the relative
Hochster--Takayama formula applies and yields
$$
H_{\mathfrak m}^{|U|+1}
\bigl(I^{(r)}/I^{(q)}\bigr)_{\mathbf a}\neq0.
$$
Since $U\cup\{2,3\}\in\Delta$,
$|U|+2\le\dim\Delta+1$, and Lemma~\ref{lem:support} gives
$$
|U|+1\le\dim\Delta<\dim\Delta+1
=\dim\bigl(I^{(r)}/I^{(q)}\bigr).
$$
This contradicts Cohen--Macaulayness. Hence, $\Delta$ is a matroid.

The final assertion follows because the first paragraph of the proof applies
to every pair $1\le a<b$.
\end{proof}

The exceptional upper exponent $q=2$ can behave very differently. The next
small example shows that the second symbolic conormal module may be
Cohen--Macaulay even though the ordinary conormal module and every ordinary or
symbolic window with upper exponent at least three are not.

\begin{example}\label{ex:symbolic-second-exception}
Let
$\Delta=\langle\{1,4\},\{2,3\},\{2,4\},\{3,4\}\rangle$ and let
$$
I=I_\Delta=(x_1x_2,x_1x_3,x_2x_3x_4)
\subseteq k[x_1,x_2,x_3,x_4].
$$
The graph $\Delta$ is connected and has diameter two, so $S/I^{(2)}$ is
Cohen--Macaulay by \cite[Corollary~2.3]{MinhTrung}. Since $S/I$ is also
Cohen--Macaulay, the exact sequence
$$
0\longrightarrow I/I^{(2)}\longrightarrow S/I^{(2)}
\longrightarrow S/I\longrightarrow0
$$
shows that $I/I^{(2)}$ is Cohen--Macaulay.

However, $\Delta$ is not a matroid: for $F=\{1\}$ and $G=\{2,3\}$,
neither $F\cup\{2\}$ nor $F\cup\{3\}$ is a face. Thus
Theorem~\ref{thm:symbolic} gives that
$I^{(r)}/I^{(q)}$ is not Cohen--Macaulay whenever
$1\le r<q$ and $q\ge3$. Moreover, if $u=x_1x_2x_3x_4$, then
$u\in I^{(2)}\setminus I^2$: the minimal primes of $I$ are
$(x_2,x_3)$, $(x_1,x_4)$, $(x_1,x_3)$, and $(x_1,x_2)$, whereas no product of
two minimal generators of $I$ divides $u$. Hence, $I^2\neq I^{(2)}$, and
Proposition~\ref{prop:second-symbolic} implies that $I/I^2$ is not
Cohen--Macaulay. Since the minimal generators $x_1x_2$ and $x_1x_3$ have
intersecting supports, $I$ is not a complete intersection. Therefore,
Theorem~\ref{thm:ordinary} gives that every ordinary window $I^r/I^q$ with
$q\ge3$ is also non-Cohen--Macaulay.
\end{example}

\end{document}